\documentclass{eptcs}
\usepackage{underscore}           

\usepackage{amssymb}
\usepackage{amsmath}
\usepackage{amsthm}
\usepackage{graphicx}
\usepackage{tikz}

\usepackage{subcaption}
\title{A Note on One Less Known Class of Generated Residual Implications}
\author{Vojt\v ech Havlena
\institute{Brno University of Technology, \\ Faculty of Information Technology, \\ Brno, Czech Republic}
\email{xhavle03@stud.fit.vutbr.cz},
\and
Dana Hlin\v en\'a
\institute{Brno University of Technology, \\ Faculty of Electrical Engineering and Communication, \\ Brno, Czech Republic}
\email{hlinena@feec.vutbr.cz}
}

\def\titlerunning{A Note on One Less Known Class of Generated Residual Implications}
\def\authorrunning{V. Havlena and D. Hlin\v en\'a}

\newtheorem{definition}{Definition}
\newtheorem{remark}{Remark}
\newtheorem{theorem}{Theorem}
\newtheorem{lemma}{Lemma}
\newtheorem{example}{Example}
\newtheorem{corollary}{Corollary}

\begin{document}
\maketitle

 \begin{abstract}

This paper builds on our contribution \cite{HaHl} which studied  modelling of the conjunction in human language. 
We have discussed three different ways of constructing a  conjunction. We have dealt with generated t-norms, generated means and Choquet integral.

In this paper we construct the residual operators based on the above conjunctions. The only operator based on a t-norm is an implication. We show that this implication belongs to the class of generated implications $I^g_N$ which was introduced in \cite{smu4} and studied in \cite{HlBi}. We study its properties. Moreover, we investigate this class of generated implications. Some important properties, including relations between some classes of implications, are given.

  \end{abstract}
  
\section{Introduction}

In \cite{HaHl}, we studied  modelling of the conjunction in human language. We have experimentally rated simple statements and their conjunctions. Then we have tried, on the basis of measured data, to find a suitable function, which corresponds to human conjunction. We have discussed three different ways of constructing a  conjunction. We have dealt with generated t-norms, generated means and Choquet integral. Now we are interested in a construction of the implications based on the above conjuctions.
One of the possible ways to construct the implications is the following transformation
$$\forall x, y, u \in [0,1]; C(x,u) \leq y \iff  R_C(x,y) \geq u.$$
This transformation produces the residual operator $R_C$ based on the given conjunction $C.$ For some conjunctions we can get, in this way, a residual operator which is an implication. 

For better understanding we recall basic definitions and statements used in the paper. 
We deal with multivalued (MV for short) logical connectives, which are monotone extensions of the classical
connectives  on the unit interval $[0,1]$. We turn our attention to the conjunctions in MV-logic. Usually, the  triangular norms are used to interpret the conjunctions in MV-logic.

\begin{definition}\cite{klement2000triangular}
A triangular norm (t-norm for short) is a binary operation on the unit interval $[0,1]$, i.e., a function $T:[0,1]^2 \to [0,1]$, such that for all $x, y, z \in [0,1]$ the following four axioms are satisfied: 
\begin{itemize}
\item (T1) Commutativity
$$T(x,y)=T(y,x),$$
\item (T2) Associativity
$$T\left (T(x,y), z\right )=T\left(x, T(y,z)\right),$$
\item (T3) Monotonicity
$$T(x,y)\leq T(x,z) \mbox{~whenever~} y \leq z,$$
\item (T4) Boundary Condition
$$T(x,1)=x.$$
\end{itemize}
\end{definition}

The four basic t-norms are:
\begin{itemize}
\item the minimum t-norm $T_M(x,y)=\min\{x,y\},$
\item the product t-norm $T_P(x,y)=x \cdot y,$
\item the \L{}ukasiewicz t-norm $T_{L}(x,y)=\max\{0,x+y-1\},$
\item the drastic product $T_D(x,y)=\begin{cases} 0& \mbox{if~} \max\{x,y\}<1,\\
\min\{x,y\}& \mbox{otherwise.}
\end{cases}$
\end{itemize}

We deal only with such continuous t-norms, that are generated by a unary function (the generator). 
One possibility is to generate by an additive generator, which is a
strictly decreasing function $f$ from the unit interval $[0,1]$  to $[0,+\infty]$ such that $f(1) = 0$ and $f(x) + f(y) \in H(f) \cup [f(0^+), +\infty]$ for all $x, y \in [0, 1]$, where $H(f)$ is range of $f$. Then the generated t-norm is given as follows

$$        T(x, y) = f^{(-1)}\left(f(x) + f(y)\right),$$
where $f^{(-1)}:[0,+\infty] \to [0,1]$ and $f^{(-1)}(y)=\sup\{x \in [0,1]\,|\,f(x)>y\}.$
Note, that $f^{(-1)}$ is a pseudo-inverse, which is a monotone extension of the ordinary inverse function.
For an illustration, we give the following example of parametric class of t-norms and their additive generators.

The family of Yager t-norms, introduced  by Ronald R. Yager, is given for $   0 \leq p \leq + \infty  $ by
$$T_p^Y(x,y)=\begin{cases}
T_D(x,y) & \mbox{if } p=0,\\
T_M(x,y) & \mbox{if } p=+ \infty,\\
\max \left\{0,1-\left((1-x)^p+(1-y)^p \right)^{\frac1p}\right\} & \mbox{if~} 0 < p < +\infty.

\end{cases}
$$

The additive generator of $T^{Y}_p$ for $0 < p < +\infty$ is
$$f^Y_p(x)=(1-x)^p.$$

Because of associativity, we can extend t-norms to the $n$-variete case as:

$$ x_T^{(n)}=\begin{cases} x &\mbox {if $n=1,$} \\ {T(x,x_T^{(n-1)})}
&\mbox {if $n>1.$}\end{cases}$$

A t-norm $T$ is called Archimedean if for each $x,y$ in the open interval $]0,1[$ there is a natural number $n$ such that $x^{(n)}_T \leq y.$
It is sufficient to investigate Archimedean t-norms, because every non-Archimedean t-norm can be approximated arbitrarily well with Archimedean t-norms, \cite{Jenei1998273,Jenei1998179}. 

\begin{remark}If $T$ is a t-norm, then the dual function $S:[0,1]^2 \to [0,1]$ defined by $S(x,y)=1-T(1-x,1-y)$ is called a t-conorm.  Its neutral element is $0$ instead of $1$, and all other conditions remain unchanged. Analogously to the case of t-norms, some classes of t-conorms can be generated by  additive generators. The additive generator for a t-conorm is a strictly increasing function $g$ from the unit interval $[0,1]$  to $[0,+\infty]$ such that $g(0) = 0$ and $g(x) + g(y) \in H(g) \cup [g(1^-), +\infty]$ for all $x, y \in [0, 1]$. The generated t-conorm is given by 
$$        S(x, y) = g^{(-1)}\left(g(x) + g(y)\right),$$ 
where $g^{(-1)}(y)=\sup\{x \in [0,1]\,|\,g(x)<y\}$.
Note that t-conorms are usually used for modelling fuzzy disjunctions. 
\end{remark}

Now, we continue  with definitions and properties of fuzzy negations.
\begin{definition} (see e.g. in  \cite{FR})
\label{1}
A function $N:[0,1]\rightarrow [0,1]$ is called a {\em fuzzy
negation} if,
for each $a,b \in [0,1]$, it satisfies the following conditions
\begin{itemize}
\item{(i)} $a < b \Rightarrow N(b) \leq N(a),$
\item{(ii)} $N(0)=1, N(1)=0.$
\end{itemize}
\end{definition}
\begin{remark} A {\em dual negation} $N^d:[0,1] \rightarrow [0,1]$ based on a negation
$N,$ is given by $N^d(x)=1-N(1-x).$
A fuzzy negation $N$ is called {\em strict}  if $N$ is strictly decreasing and continuous for
arbitrary $x,y \in [0,1].$
 In  classical logic we have that $(\mbox{\boldmath{$A$}}')'=\mbox{\boldmath{$A$}}$. 
 In multivalued logic this equality is not satisfied for every negation.
 The negations with this equality are called {\em involutive  negations.}
   The strict negation is {\em strong} if and only if  it is involutive.
   The most important and most widely used strong negation is the standard negation $N_S(x)=1-x.$
   \end{remark}

In the literature, one can find several different definitions of
fuzzy implications. In this paper we will use the following one,
which is equivalent to the definition introduced by Fodor and
Roubens in \cite{FR}. 
\begin{definition}\label{impl}
A function $I:[0,1]^2 \rightarrow [0,1]$ is called a {\em fuzzy
implication} if it satisfies the following conditions:
\begin{itemize}
\item [(I1)]  $I$ is non-increasing in its first variable,
\item [(I2)]  $I$ is non-decreasing in its second variable,
\item [(I3)]  $I(1,0)=0$, $I(0,0)=I(1,1)=1$.
\end{itemize}
\end{definition}

We recall definitions of some important properties of fuzzy
implications which we will investigate.
\begin{definition}\label{propimpl}
A fuzzy implication $I:[0,1]^2 \rightarrow [0,1]$ satisfies:
\begin{itemize}
\item [(NP)] the left neutrality property if
$$I(1,y)=y\quad\mbox{  for all   } y \in [0,1],$$
\item [(EP)] the exchange principle if
$$I(x,I(y,z))=I(y,I(x,z))\quad \mbox{  for all   } x,y,z \in [0,1],$$
\item [(IP)] the identity principle if
$$I(x,x) = 1\quad \mbox{  for all   }x \in [0,1], $$
\item [(OP)] the ordering property if
$$x \leq y \iff I(x,y) =1\quad \mbox{  for all   } x,y \in [0,1],$$
\item [(CP)] the contrapositive symmetry with respect to a given fuzzy
negation $N$ if
$$ I(x,y)=I(N(y),N(x))\quad \mbox{  for all   } x,y \in [0,1].$$
\end{itemize}
\end{definition}
\begin{definition} Let $I:[0,1]^2 \rightarrow [0,1]$ be a fuzzy implication. The function $N_I$
defined by $N_I(x)=I(x,0)$ for all $x \in [0,1]$, is called the
natural negation of $I$.
\end{definition}

$(S,N)$-implications which are based on $t$-conorms and fuzzy
negations form one of the well-known classes of fuzzy implications.
\begin{definition} A function $I: [0,1]^2 \to [0,1]$ is called an
$(S,N)$-implication if there exist a t-conorm $S$ and a fuzzy
negation $N$ such that
$$I(x,y)=S(N(x),y), ~~~ x,y \in [0,1].$$
If $N$ is a strong negation then $I$ is called a strong implication.
\end{definition}
The following characterization of $(S,N)$-implications is from \cite{bac2}.
\begin{theorem} \label{baczinski1} (Baczy\' nski and Jayaram \cite{bac2},
Theorem 5.1)
For a function $I:[0,1]^2 \to [0,1],$ the following statements are
equivalent:
\begin{itemize}
\item $I$ is an $(S,N)$-implication generated from some t-conorm and some
continuous (strict, strong) fuzzy negation $N.$
\item $I$ satisfies (I2), (EP), and $N_I$ is a continuous (strict, strong)
fuzzy negation.
\end{itemize}
\end{theorem}
Another way of extending the classical binary implication 
to the unit interval $[0,1]$ is based on the residual operator
with respect to a left-continuous triangular norm $T$
$$I_T(x,y)=\max\{z \in [0,1]\ |\ T(x,z) \leq y\}.$$
Elements of this class are known as $R$-implications.
The following characterization of $R$-implications is from \cite{FR}.
\begin{theorem} \label{bacz2} (Fodor and Roubens \cite{FR}, Theorem
1.14)
For a function $I:[0,1]^2 \to [0,1],$ the following statements are
equivalent:
\begin{itemize}
\item $I$ is an $R$-implication based on some left-continuous t-norm $T.$
\item $I$ satisfies (I2), (OP), (EP), and $I(x,.)$ is  right-continuous for
any $x \in [0,1]$.
\end{itemize}
\end{theorem}
At last we introduce a characterization of implications based on $\Phi$-conjugate from~\cite{bac2}.
\begin{definition}
 We denote  by $\Phi$ the family of all increasing bijections on the unit interval $[0,1].$ We say that implications $I_1, I_2: [0,1]^2 \rightarrow [0,1]$ are
 $\Phi$-conjugate if there exists a bijection $\varphi\in\Phi$ such that
 $I_2 = (I_1)_\varphi$, where
 $$
 	(I_1)_\varphi(x,y) = \varphi^{-1}(I_1(\varphi(x), \varphi(y))),
 $$
 for all $x,y\in[0,1]$.
\end{definition}
\begin{theorem}\label{the:conjugate}(Baczy\' nski and Jayaram \cite{bac2},
Theorem 2.4.20)
 Let $I: [0,1]^2\rightarrow [0,1]$ be a function. Then $I$ is a continuous 
 function satisfying (OP), (EP), if and only if, $I$ is $\Phi$-conjugate
 with the \L ukasiewicz implication.
\end{theorem}

It is well-known that it is possible to generate t-norms from  one variable
functions.
Therefore the question whether something similar is possible in the case of fuzzy implications is very interesting. In \cite{Yag} Yager introduced two new classes of fuzzy
implications: $f$-implications and $g$-implications where their generators $f$ 
are continuous additive generators of continuous Archimedean t-norms  and generators $g$
are continuous additive generators of continuous
Archimedean t-conorms.

In this paper we deal with some of less known classes of generated fuzzy
implications which were introduced in \cite{smu4} and studied in \cite{HlBi}.

The first class of generated implications is based on strictly increasing functions~$g.$
\begin{theorem}\label{implg}\cite{smu4}
Let  $g:[0,1]\rightarrow [0,\infty]$ be a strictly increasing function such
that $g(0)=0$. Then the function $I^g:[0,1]^2 \rightarrow [0,1]$ which
is given by
\begin{equation}
I^g(x,y)=g^{(-1)}(g(1-x)+g(y)),
\end{equation}
is a fuzzy implication.
\end{theorem}
The fuzzy implication $I^g$ can be generalized. This generalization is based on 
replacing  the standard negation by an arbitrary one. 
\begin{theorem}\cite{smu4}
Let $g: [0,1] \rightarrow [0, \infty]$ be a strictly increasing
function such that $g(0)=0$ and $N$ be a fuzzy negation. Then the function $I_N^g$
\begin{equation}\label{g}
I_N^g(x,y)=g^{(-1)}(g(N(x))+g(y)),
\end{equation}
is a fuzzy implication.
\end{theorem}


\section{The residual operators based on the considered conjunctions}

As mentioned in the first section, we have found residual operators of conjunctions which were based on empirically measured data. 

The first conjunction was the t-norm $T^Y_2$ which is given by
$$
  T^Y_2(x,y) = \max\left\{0, 1 - \left((1-x)^2 + (1-y)^2\right)^{\frac{1}{2}}\right\}.
$$
It is Yager's t-norm with parameter $p = 2$. 
The corresponding residual operator (Fig.~\ref{fig:restnorm}) is given by 
\begin{equation}
   I_{T^Y_2}(x,y) = 1 - (\max((1-y)^2 - (1-x)^2), 0)^{\frac{1}{2}}.
\end{equation}
In general, residual implications which are based on Yager t-norms  $T^Y_{p}$ are given by:
\begin{equation}
   I_{T^Y_p}(x,y) = 1 - (\max((1-y)^p - (1-x)^p), 0)^{\frac{1}{p}}.
\end{equation}
Now, we will investigate properties of implications $I_{T^Y_p}$ and their membership in the classes of implications. 
We turn our attention to the class of $I^g$ implications.
The boundary conditions for  $I^g$ implications are given by
$$   I^g(x, 0) = g^{(-1)}\circ g(1-x) = 1-x, $$
$$   I^g(1, y) = g^{(-1)}\circ g(y) = y.$$

On the other hand, residual implication $ I_{T^Y_p}$ satisfies the following equality
$$
   I_{T^Y_p}(x, 0) = 1 - (\max(1 - (1-x)^p), 0)^{\frac{1}{p}} =
    1 - (1 - (1-x)^p)^{\frac{1}{p}}.
$$
Therefore the implication $ I_{T^Y_p}$  can not be expressed as $I^g$, but as  $I^g_N$.
The function
$$N_p(x) = I_{T^Y_p}(x, 0) = 1 - (1 - (1-x)^p)^{\frac{1}{p}}$$
is a negation (particularly, for $p = 2$ we get $N_2(x) = 1- \sqrt{x(2-x)} )$ and
since $$I^g_N(x,0) = g^{(-1)}(g(N(x)), g(0)) = N(x),$$ the implication $I_{T^Y_p}$ is expressed by the  function $I^g_N$ with negation $N = N_p$.

Furthermore, we consider the function $g_p(x) = 1 - (1 - x)^p$, where
$g_p^{-1}(x) = 1 - (1-x)^{\frac{1}{p}}$. Then the function $I^{g_p}_{N_p}$ is given by
\begin{eqnarray*}
  I^{g_p}_{N_p}(x,y) &=& g_p^{-1}(\min(g_p(N_p(x))+g_p(y), g_p(1))) \\ &=&
    g_p^{-1}(\min((1-x)^p+1 - (1-y)^p, 1)) \\ &=&
    1 - (1 - \min((1-x)^p+1 - (1-y)^p, 1))^{\frac{1}{p}}.
\end{eqnarray*}
Since $1 - \min(1 - x, 1 - y) = \max(x,y)$ we have
$$
  I^{g_p}_{N_p}(x,y) = 1 - (\max((1-y)^p - (1-x)^p), 0)^{\frac{1}{p}} = I_{T^Y_p}(x,y).
$$


Let $p >0$. Directly from Definition \ref{propimpl} we get that the implications $I_{T^Y_p}$ satisfy properties (IP) and (NP). Since the  implications $I_{T^Y_p}$ are residual operators based on the left-continuous t-norms $T^Y_p$, and due to Theorem \ref{bacz2},  properties  (EP) and (OP) are satisfied for these implications.  Additionally
$$
  I_{T^Y_p}(N_p(y), N_p(x)) = 1 - (\max(1-(1-x)^p - (1-(1-y))^p), 0)^{\frac{1}{p}} = I_{T^Y_p}(x,y),
$$
which is the property (CP) with respect to the negations $N_p$.

The next conjunction is a quasi-arithmetic mean $M$ (for more details see \cite{HaHl}). Its residual operator is given by formula
\begin{eqnarray*}
  M_r(x,y) &=& \sup\{t\in[0,1]\ |\ M(x,t) \leq y\} = 
  \sup\left\{t\in[0,1]\ \Big|\ \frac{1}{2}(x^2 + t^2) \leq y^2\right\} \\ &=& 
  (\min\{\max\{ 2y^2 - x^2, 0 \}, 1\})^{\frac{1}{2}}.
\end{eqnarray*}
This operator is not an implication, since the boundary condition $I(0,0)=1$ is violated (Fig.~\ref{fig:resmean}). The same problem occurs with residual operator of the last conjunction, which is Choquet integral (Fig.~\ref{fig:rescho}). Therefore we will not discuss these operators.

\noindent
\begin{figure}[h!]
  \centering
  \begin{subfigure}[b]{0.48\textwidth}
    \centering
    \includegraphics[scale=0.32]{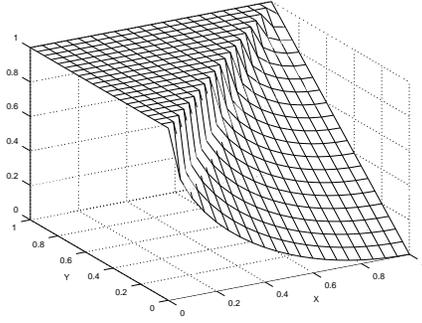}
    \caption{Residual operator of t-norm $T^Y_2$.}
    \label{fig:restnorm}
  \end{subfigure}
~
  \begin{subfigure}[b]{0.49\textwidth}
    \includegraphics[scale=0.32]{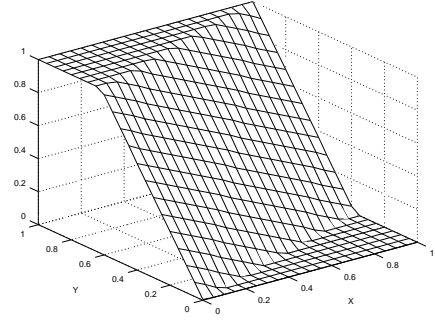}
    \caption{Residual operator of Choquet integral.}
    \label{fig:rescho}
  \end{subfigure}
~
  \begin{subfigure}[b]{0.9\textwidth}
    \centering
    \includegraphics[scale=0.32]{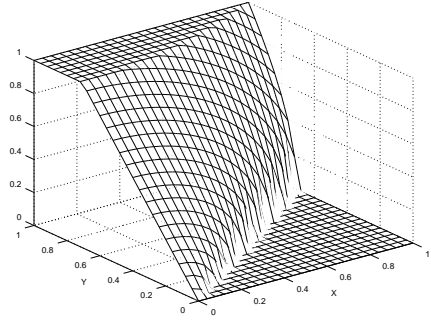}
    \caption{Residual operator of  quasi arithmetic mean $M_r$.}
    \label{fig:resmean}
  \end{subfigure}
  \caption{Residual operators based on the found conjunctions.}
\end{figure}

\section{Properties of $I^g$ and $I^g_N$ implications}
In this section we investigate properties of generated implications $I^g$ and $I^g_N$.
We focus on relations between these generated implications and some well known classes
of implications.

In the following text we denote by $\mathbb{I}^g$  the class of $I^g$ implications and by $\mathbb{I}^g_N$
the class of $I^g_N$ implications. Further we denote by $\mathbb{I_{T_{LC}}}$  
the class of $R$-implications based on left-continuous t-norm and by $\mathbb{I_{S, N}}$
the class of $(S, N)$-implications. With the subscript $c$ we denote a continuous
function (we use it in the context of continuous functions $g$ and $N$). 



Two of the best known classes of implications are $R$-implications and $(S,N)$-implications. 
In the first part we focus on the relation of the classes $\mathbb{I}^{g}_{N}$ and $\mathbb{I_{S, N}}$.
We are interested in two questions\,--\,whether the class $\mathbb{I}^{g}_{N}$ is a proper subclass
of $\mathbb{I_{S, N}}$ and if not, find a subclass $C$ of $\mathbb{I}^{g}_{N}$ satisfying
$C\subseteq \mathbb{I_{S, N}}$.

\begin{lemma}\label{lem:cont}
  Let $I: [0,1]^2\rightarrow [0,1]$ be an implication. If $I \in \mathbb{I}^{g_c}_{N}$ then 
  $I \in \mathbb{I_{S, N}}$.
\end{lemma}
\begin{proof}
We deal with $I^g_N$, where $N$ is an arbitrary negation and $g$ is a continuous generator. Since $g$ is a strictly 
  increasing continuous function with $g(0) = 0$, it holds
  $$
    g^{(-1)}(g(x) + g(y)) = S_g(x,y),
  $$
  where $S_g$ is t-conorm generated by $g$. Accordingly
  $$
    I(x,y) = I^g_N(x,y) = g^{(-1)}(g(N(x)) + g(y)) = S_g(N(x), y) 
  $$
  and thus $I \in \mathbb{I_{S, N}}$.
\end{proof}

For illustration we provide the following example:
\begin{example}
\label{g-cka} Let  $g:[0,1] \rightarrow [0,\infty]$ be a function 
given by the following formula
$$g(x)=-\ln(1-x).$$
The function $g$ is strictly increasing and continuous.
Its pseudoinverse function $g^{(-1)}$ is given by
$$g^{(-1)}(x)=1-e^{-x}\quad \mbox{for $x \in [0, \infty]$}.$$
Then for the function $g$ we get the following implication
$$I^{g}(x,y)=1-e^{\ln(x(1-y))}=1-x+xy,$$
which is $S_P(1-x,y),$ where $S_P$ is dual operator to the product t-norm and $I^g$ is thus an $(S, N)$-implication with negation $N(x)=1-x.$
\end{example}

\begin{lemma}
 For the classes  $\mathbb{I}^{g}_{N} $ and $\mathbb{I_{S,N}},$ it holds $\mathbb{I}^{g}_{N} \setminus \mathbb{I_{S,N}} \neq \emptyset$.
\end{lemma}
\begin{proof} We assume  $\mathbb{I}^{g}_{N} \setminus \mathbb{I_{S,N}} = \emptyset$.

We turn our attention to the following example:
We consider the strictly increasing function  $f:[0,1] \rightarrow [0,\infty]$ which is given by formula
$$f(x)=\begin{cases}  x    &\mbox{if $x \leq 0.5$}, \\
0.5+0.5x      &\text{otherwise.} \end{cases}$$
Its pseudoinverse function is given by
 $$f^{(-1)}(x)=\begin{cases} x &\mbox{if $x \leq 0.5$}, \\
0.5   &\mbox{if $0.5 < x \leq 0.75$}, \\
2x-1   &\mbox{if $0.75 < x \leq 1$}, \\
1   &\mbox{if $1 < x$}. \end{cases}$$
Finally, for implication based on the function $f$ we get
$$I^{f}(x,y)=\begin{cases}  1-x+y      &\mbox{if $x \geq 0.5, y \leq 0.5, x-y \geq 0.5$}, \\
0.5      &\mbox{if $x \geq 0.5, y \leq 0.5, 0.25 \leq x-y < 0.5$}, \\
1-2x+2y      &\mbox{if $x \geq 0.5, y \leq 0.5, x-y < 0.25$}, \\
\min(1-x+2y,1)       &\mbox{if $x < 0.5, y \leq 0.5$}, \\
\min(2-2x+y,1)       &\mbox{if $x \geq 0.5, y>0.5$}, \\
1       &\mbox{if $x < 0.5, y > 0.5.$} \end{cases}$$
Now we will construct a negation $N$ and a t-conorm $S$ such that $I^{f}(x,y)=S(N(x), y).$
 From the boundary condition we get
 $$
  I^{f}(x,0) = f^{(-1)}\circ f(1-x) = 1-x = S(N(x), 0) = N(x)
 $$
 and therefore $S(x,y) = I^{f}(1-x,y)$ is a t-conorm.
 But 
 \begin{eqnarray*}
  S(0.3, S(0.35, 0.2)) &=& S(0.3, 0.5) = 1-1.4 + 1 = 0.6 \\
  S(S(0.3, 0.35), 0.2) &=& S(0.5, 0.2) = 0.5
 \end{eqnarray*}
 and thus $S$ is not associative, which is a contradiction.
\end{proof}

\begin{theorem}
 For the classes $\mathbb{I}^{g_c}, \mathbb{I}^{g_c}_{N_c}$ and $ \mathbb{I_{S,N}},$ it holds $\mathbb{I}^{g_c} \subset \mathbb{I}^{g_c}_{N_c} \subset \mathbb{I_{S,N}}$.
\end{theorem}
\begin{proof}
  Apparently $\mathbb{I}^{g_c} \subseteq \mathbb{I}^{g_c}_{N_c}$ holds true and the implication $I_{T^Y_2}$
  from the previous section forms an example of an implication in $\mathbb{I}^{g_c}_{N_c} \setminus \mathbb{I}^{g_c}$.
  From Lemma~\ref{lem:cont} we get $\mathbb{I}^{g_c}_{N_c} \subseteq \mathbb{I_{S, N}}$.
  If we consider the $(S,N)$-implication $I(x,y) = \max\{1-x, y\}$ and try to express this implication as $I^g_N$, we obtain
  $I(x,y) = \max\{1-x, y\} = g^{(-1)}(g(1-x) + g(y))$, which is an expresion via additive generator, but the t-conorm $\max\{x, y\}$ has no additive
generator. Therefore $\mathbb{I_{S,N}}\setminus \mathbb{I}^{g_c}_{N_c} \not = \emptyset.$
\end{proof}

The second part is devoted to the relation of a subclass of $\mathbb{I}^{g}_{N}$, with continuous generator $g$ and continuous negation $N$, and $\mathbb{I_{T_{LC}}}$, which is explained in the following assertion.

\begin{lemma}\label{lem:rconj}
  Let $I: [0,1]^2\rightarrow [0,1]$ be an implication such that $I \in\mathbb{I}^{g_c}_{N_c}$.
  Then $I$ is an $R$-implication based on left-continuous t-norm if and only if $I$ is $\Phi$-conjugate with
  the \L ukasiewicz implication.
\end{lemma}
\begin{proof}
 $(\Rightarrow)$ We assume that $I = I^g_N$ for some continuous $g$ and $N$. According to
 Lemma~\ref{lem:cont} we get $I(x,y) = S_g(N(x), y)$. Since both $g$ and $N$ are continuous
 functions, also $S_g$ is continuous and therefore $I$ is continuous, too. By the assumption, $I$
 is an  $R$-implication based on left-continuous t-norm. From Theorem~\ref{bacz2} we directly
 get that, $I$ satisfying properties (OP) and (EP) and from Theorem~\ref{the:conjugate}
 we finally obtain that $I$ is $\Phi$-conjugate with the \L ukasiewicz implication.

 
 $(\Leftarrow)$ Since $I$ is $\Phi$-conjugate with the \L ukasiewicz implication, according to 
 Theorem~\ref{the:conjugate}, $I$ is a continuous implication satisfying (OP), (EP) and 
 from Theorem~\ref{bacz2} we get that $I$ is an $R$-implication based on a left-continuous t-norm.
\end{proof}

\begin{theorem}
  Let $I: [0,1]^2\rightarrow [0,1]$ be an implication such that $I\in\mathbb{I}^{g_c}_{N_c}$. 
  Then $I$ is an $R$-implication based on a left-continuous t-norm if and only if
  $I = I^\varphi_{N_\varphi}$, where $N_\varphi(x) = \varphi^{-1}(1 - \varphi(x))$ for some $\varphi \in \Phi$.
\end{theorem}
\begin{proof}
 $(\Rightarrow)$ Since $I$ is an $R$-implication based on a left-continuous t-norm, from Lemma~\ref{lem:rconj}
 we get that $I$ is $\Phi$-conjugate with the \L ukasiewicz implication, and thus for all $x,y\in[0,1],$
 $$
    I(x,y) = (I_{\mathbf{LK}}(x,y))_\varphi = \varphi^{-1}(\min\{1-\varphi(x)+\varphi(y), 1\}) = I^\varphi_{N_\varphi}(x,y),
  $$
  where $I_{\mathbf{LK}}$ is the \L ukasiewicz implication given by $I_{\mathbf{LK}}(x,y) = \min\{1 - x+y, 1\}$.
  The last equality holds because, for all $x,y\in[0,1]$
   $$
    I^\varphi_{N_\varphi}(x,y) = \varphi^{-1}(\min\{\varphi(N_\varphi(x)) + \varphi(y), \varphi(1)\}) = 
      \varphi^{-1}(\min\{1 - \varphi(x) + \varphi(y), 1\}).
  $$
  $(\Leftarrow)$ This directly follows from  Lemma~\ref{lem:rconj} and equality $(I_{\mathbf{LK}})_\varphi = I^\varphi_{N_\varphi}$. 
\end{proof}

Directly from previous theorem we get what are the intersection of $\mathbb{I_{T_{LC}}}$ and $\mathbb{I}^{g_c}_{N_c}$,
$\mathbb{I}^{g_c}$ respectively.
(Fig.~\ref{fig:intersectionRIG}).

\begin{corollary}
 $\mathbb{I_{T_{LC}}} \cap \mathbb{I}^{g_c}_{N_c} = \mathbb{I}^{g_\varphi}_{N_\varphi}$, where 
 $\mathbb{I}^{g_\varphi}_{N_\varphi} = \{ I^\varphi_{N_\varphi}\ |\ \varphi\in\Phi\}$.
\end{corollary}
\begin{corollary}
 $\mathbb{I_{T_{LC}}} \cap \mathbb{I}^{g_c} = \mathbb{I}^{g_\varphi}$, where 
 $\mathbb{I}^{g_\varphi} = \{ I^\varphi\ |\ \varphi\in\Phi, \varphi(x) + \varphi(1-x) = 1, x\in[0,1]\}$.
\end{corollary}

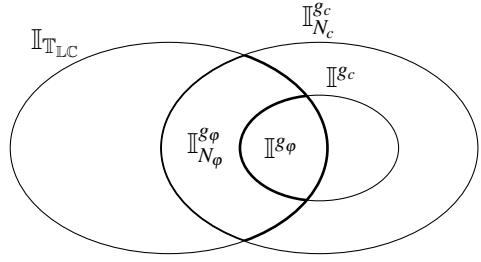
\begin{figure}
\centering
\begin{tikzpicture}
  \begin{scope}
    \clip (2,0) ellipse (60pt and 40pt);
    \draw [line width=1pt] (0,0) ellipse (60pt and 40pt);
    \clip (0,0) ellipse (60pt and 40pt);
    \draw [line width=1pt] (2,0) ellipse (60pt and 40pt);
    \draw [line width=1pt] (2,0) ellipse (30pt and 20pt);
  \end{scope}

  \draw (0,0) ellipse (60pt and 40pt);
  \draw (2,0) ellipse (60pt and 40pt);
  \draw (2,0) ellipse (30pt and 20pt);

  \draw (-1.5,1.4) node {$\mathbb{I_{T_{LC}}}$};
  \draw (2,1.7) node {$\mathbb{I}^{g_c}_{N_c}$};
  \draw (2.3,0.95) node {$\mathbb{I}^{g_c}$};
  \draw (0.5,0) node {$\mathbb{I}^{g_\varphi}_{N_\varphi}$};
  \draw (1.5 ,0) node {$\mathbb{I}^{g_\varphi}$};
\end{tikzpicture}
\caption{Intersection of the class of $R$-implications based on left-continuous t-norm and the class of $I^g_N$ implications with 
continuous generator $g$ and negation $N$.}
\label{fig:intersectionRIG}
\end{figure}

\section{Conclusion}
We have investigated the residual operator of the conjunction. This conjunction was based on empirical data. It turned out that the only operator based on generated t-norm is  an implication and it belongs to the less known class of generated implications $I^g_N$ where $N(x) \not = N_S(x)$.  We have studied the properties of $I^g_N$-implications. We showed that although the classes $I^g_N$  and $(S, N)$-implications are similar, they are not identical. And also, we examined the relationship between classes  $I^g_N$ and $R$-implications based on left-continuous t-norms. In the future we plan to model implications in human language via fitting residual operators to empirical data.

\paragraph*{Acknowledgement.} The work was supported by the BUT project FIT-S-14-2486.

\bibliographystyle{eptcs}
\bibliography{literature}

\end{document}